\documentclass[12pt,a4paper]{amsart}
\usepackage[english]{babel}

\usepackage{amssymb} 
\usepackage{mathtools}
\usepackage{amsthm} 
\usepackage{amsmath}
\usepackage{graphicx} 
\usepackage[all]{xy} 
\usepackage{wasysym}
\usepackage{wrapfig} 
\usepackage{xcolor} 
\usepackage{verbatim}
\usepackage{tikz}
\usepackage{tikz-cd}
\usepackage{calc}
\usetikzlibrary{calc}

\usepackage{thmtools}

\usepackage{csquotes}
\usepackage{hyperref}
\hypersetup{
    final=true,
    plainpages=false,
    pdfstartview=FitV,
    pdftoolbar=true,
    pdfmenubar=true,
    bookmarksopen=true,
    bookmarksnumbered=true,
    breaklinks=true,
    linktocpage,
    colorlinks=true,
    linkcolor=teal,
    urlcolor=teal,
    citecolor=teal,
    anchorcolor=green
  }

\usepackage[
style = alphabetic, 
sorting=nyt, 
url=true,
maxnames=10,
minnames=10, 
maxalphanames=10]{biblatex}
\DeclareFieldFormat[article]{volume}{\mkbibbold{#1}\addcolon \addspace}
\DeclareFieldFormat[article]{number}{\mkbibparens{#1}}
\renewbibmacro*{in:}{}

\numberwithin{equation}{section}

\renewcommand{\to}{\longrightarrow}

\newcommand{\R}{{\mathbb{R}}}

\newcommand{\so}{\mathfrak{so}}

\newcommand{\n}{\mathfrak{n}}

\newcommand{\h}{\mathfrak{h}}
\newcommand{\m}{\mathfrak{m}}
\newcommand{\g}{\mathfrak{g}}

\newcommand{\RH}{\mathbb{R}\mathrm{H}}

\theoremstyle{definition}
\newtheorem*{AS-theorem}{Ambrose-Singer Theorem}
\newtheorem{definition}{}[section]

\newtheorem{proposition}[definition]{}

\newtheorem{theorem}[definition]{}

\newtheorem{corollary}[definition]{}

\newtheorem{lemma}[definition]{}

\newtheorem{remark}[definition]{}

\newtheorem{problem}[definition]{}

\renewcommand{\eqref}[1]{\hyperref[#1]{(\protect\NoHyper\ref{#1}\protect\endNoHyper)}}

\DeclareCiteCommand{\cite}
{}
{\bibhyperref{[{\protect\NoHyper\usebibmacro{cite}}\protect\endNoHyper\usebibmacro{postnote}]}}
{   
    \multicitedelim%
}
{}

\usepackage[a4paper]{geometry}
\title{Canonical reductive decomposition of extrinsic homogeneous submanifolds}

\author{José Luis  {\small CARMONA JIMÉNEZ}}
\address{JLCJ: Institute of Mathematics “Simion Stoilow” of the Romanian Academy, 21 Calea Grivitei, 010702 Bucharest, Romania.}
\email{jcarmona@imar.ro}
\author{Marco {\small Castrillón López*}}
\address{MCL: Dept. Álgebra, Geometría y Topología, Facultad de Ciencias Matemáticas, Universidad Complutense de Madrid, 28040 Madrid, Spain.}
\email{mcastri@mat.ucm.es}

\date{\today}

\begin{document}
\thanks{
\noindent * = Corresponding author.
\newline \noindent Both authors have been partially supported by projects PID2021-126124NB-I00  and PID2024-156578NB-I00, Agencia Estatal de Investigación, Spain. JLCJ has been supported by the PNRR-III-C9-2023-I8 grant CF 149/31.07.2023 {\em Conformal Aspects of Geometry and Dynamics}.}

\begin{abstract}
    Let $\overline{M}=\overline{G}/\overline{H}$ be a homogeneous Riemannian manifold. Given a Lie subgroup $G\subset \overline{G}$ and a reductive decomposition of the homogeneous structure of $\overline{M}$, we analyze a canonical reductive decomposition for the orbits of the action of $G$. These leaves of the $G$-action are extrinsic homogeneous submanifolds and the analysis of the reductive decomposition of them is related with their extrinsic properties. We connect the study with works in the literature and initiate the relationship with the Ambrose-Singer theorem and homogeneous structures of submanifolds.
\end{abstract}

\maketitle

{\footnotesize

}
\section{Introduction}

Two classical topics in Differential Geometry are those of homogeneous manifolds and the geometry of submanifolds. The combination of both has been one main research interest in classical and recent articles. More precisely, given a homogeneous ambient space where a group $\overline{G}$ is acting transitively, the objective is to analyze the geometry of the leaves induced by the action of a subgroup $G\subset\overline{G}$. Furthermore, characterization of those submanifolds of $\overline{M}$ that are leaves of any action is an essential question.

The geometry of homogeneous spaces when they are reductive is particularly rich because they are equipped with a canonical connection that reflects deeply the symmetries of these spaces. These canonical connections, defined through reductive decompositions, provide fundamental tools to analyze curvature properties and symmetry groups of homogeneous manifolds. Furthermore, the canonical connection is the main character of the Ambrose-Singer theorem: In the Riemannian case, a manifold $(\overline{M},\overline{g})$ is locally homogeneous if and only if there is a connection $\tilde{\nabla}$ such that $\tilde{\nabla}\tilde{R}=0$, $\tilde{\nabla} \tilde{T}=0$ and $\tilde{\nabla}\overline{g}=0$, where $\tilde{R}$ and $\tilde{T}$ are the curvature and torsion of $\tilde{\nabla}$. We can upgrade local homogeneity to global homogeneity if some topological conditions are added (connectedness, simply connectedness, and completeness). A version for manifolds equipped with geometry defined by tensors (non-necessarily Riemannian) can be found in \cite{CC2022*}).

In this paper, our objects to study are reductive extrinsically homogeneous submanifolds of homogeneous manifold $\overline{M} = \overline{G}/\overline{H}$, that is, closed submanifolds $M \subset \overline{M}$ that are orbits $M = G \cdot o$ ($o\in M$) of a closed subgroup $G \subset \overline{G}$. These spaces are homogeneous spaces $M=G/H$, with $H = G \cap \overline{H}$, and admit reductive decomposition of the corresponding Lie algebra $\mathfrak{g} = \mathfrak{h} + \mathfrak{m}$. A classical reference on extrinsically homogeneous submanifolds is \cite{E1998} and since its analysis is performed in non-necessarily Riemannian ambient spaces, reductivity (which is taken as an assumption) is not guaranteed neither for $\overline{M}$ nor $M$. In particular, \cite[Thm.~2]{E1998} aligns each reductive decomposition of the homogeneous submanifold with a $G$-connection on $E = T \overline{M} \vert _{M}$ which can be thought of as canonical since it comes from the choice of a reductive decomposition on $M$. When we focus the attention to the Riemannian setting in the literature, the analysis have been mainly developed in specific geometric contexts (for example, space forms, or symmetric spaces, see \cite{quast}, \cite{tillman}). In this article, we investigate the relationship between reductive decompositions of and arbitrary Riemannian homogeneous manifold $\overline{M}$ and reductive decompositions in the leaves $M$ and, in particular, we propose a canonical one in these submanifolds starting from every initial decomposition of the ambient manifold induced by the metric structure.

The article is organized as follows: In Section~\ref{sec:2}, we recall basic definitions and known results concerning reductive extrinsically homogeneous submanifolds. In Section~\ref{sec:Riemannian-Case}, we analyze the particular setting of homogeneous Riemannian manifolds, obtaining a natural reductive decomposition of an extrinsically homogeneous submanifold from the given decomposition of the ambient space. Furthermore, we investigate the geometric properties induced by such decompositions. In Section~\ref{sec:hom-structures}, we define the concept of Riemannian homogeneous structures on submanifolds of homogeneous spaces, this definition generalizes the known definition for homogeneous structures of spaces forms, see~\cite[Rmk.~6.1.5]{BCO2016}. We prove that reductive extrinsically homogeneous submanifolds naturally admit such structures. Moreover, we propose and discuss an open problem regarding minimal conditions ensuring homogeneity of a submanifold endowed with a homogeneous structure. Finally, Section~\ref{sec-examples} provides explicit examples illustrating the general theory developed in previous sections, including the canonical connections for submanifolds such as horospheres in hyperbolic spaces and concentric spheres in Euclidean spaces.

\section{Extrinsically Homogeneous Submanifolds}\label{sec:2}
Let $\overline{M} = \overline{G}/\overline{H}$ be a homogeneous manifold and let $\overline{\g}$ and $\overline{\h}$ be the Lie algebras of $\overline{G}$ and $\overline{H}$, respectively. A homogeneous manifold $\overline{M}$ is said to be \emph{reductive} if there exists an $\mathrm{Ad}(\overline{H})$-invariant subspace $\overline{\mathfrak{m}}$ such that $\overline{\mathfrak{g}} = \overline{\mathfrak{h}} + \overline{\mathfrak{m}}$.
If we denote $o=[e]_{\overline{H}}$, then the linear map
\begin{align*}
    \varphi \colon \overline{\mathfrak{g}}\;&\longrightarrow\; T_{o}\overline{M},\\
X \;&\longmapsto\; \left.\frac{d}{dt}\right|_{t=0}\exp(tX)\cdot o
\end{align*}
has kernel \(\overline{\mathfrak{h}}\), so that if \(X_1,X_2 \in \overline{\mathfrak{g}}\)
satisfy
\[
\left.\frac{d}{dt}\right|_{t=0}\exp(tX_1)\cdot o \;=\;
\left.\frac{d}{dt}\right|_{t=0}\exp(tX_2)\cdot o,
\]
then \(X_1 - X_2 \in \overline{\mathfrak{h}}\). In particular, the restriction
\(\varphi_{\overline{\mathfrak{m}}}\) of \(\varphi\) to \(\overline{\mathfrak{m}}\),
\begin{equation}\label{eq:Identificación}
    \varphi_{\overline{\mathfrak{m}}} \colon \overline{\mathfrak{m}} \;\longrightarrow\; T_{p_0}\overline{M},
\end{equation}
is a bijection. A reductive decomposition determines a canonical linear connection $\tilde{\nabla}$ which is characterized by the condition 
\begin{equation}
    \label{carcon}
  (\tilde{\nabla}_{X^*} B^*)_o \;=\; -\,[X,\,B]^*_o,
\end{equation}
for $B \in \overline{\mathfrak{g}}$, $X \in \overline{\mathfrak{m}}$, where the star stands for the fundamental vector fields, i.e., $X^*_p= \left.\frac{d}{dt}\right|_{t=0} \mathrm{exp}(tX)\cdot p$, $X\in\overline{\mathfrak{g}}$. 

We consider a Lie subgroup $G\subset\overline{G}$ such that the orbit $M=G\cdot o$ is a closed submanifold. The homogeneous space $M=G/H$, with $H=\overline{H}\cap G$, is said to be a \emph{reductive orbit} or a \emph{reductive extrinsically homogeneous submanifold} (known in~\cite{E1998} simply as \emph{reductive homogeneous}) in $\overline{M}$ if there exists an $\mathrm{Ad} (H)$-invariant subspace $\m$ such that $\g = \h + \m$. Furthermore, a linear connection 
\(D\) on \(E = T\overline{M}\vert_M\) is called a \emph{\(\overline{G}\)-connection} if, for every 
piecewise smooth curve \(c: [a,b] \to M\), the \(D\)-parallel transport 
\(\tau_c : T_{c(a)}\overline{M} \to T_{c(b)} \overline{M}\) is determined by some element of \(\overline{G}\). In other 
words, there exists some \(g \in \overline{G}\) such that, for all 
\(v \in T_{c(a)} \overline{M}\), it holds that
\[
  \tau_c(v) = (L_g)_* v.
\]
We recall the following result of~\cite[Thm.~2]{E1998}.

\begin{theorem}\label{thm:Esc}
Let $\overline{M}$ be a reductive homogeneous manifold with canonical connection $\tilde{\nabla}$. A closed submanifold $M \subset \overline{M}$ is reductive extrinsic homogeneous if and only if 
there exists a linear $\overline{G}$-connection $D$ on $E := T\overline{M} \vert_{M}$ such that 
$TM \subset E$ is a parallel subbundle and the tensor 
\(
\Gamma = \tilde{\nabla}- D \colon TM \longrightarrow \mathrm{End}(E)
\)
is parallel with respect to $D$.
\end{theorem}

The connection $D$ above is not unique and comes from a choice of a reductive decomposition $\mathfrak{g}=\h + \m$. In fact, $D$ can be thought of as canonical connection with respect to that decomposition. Indeed, the left-invariant distribution of subspaces defined by $\m$ is a principal connection for the principal bundle $G\to M=G/H$. The canonical connection on $M$ is the induced linear connection by this principal connection when $TM$ is regarded as the associated vector bundle with respect to the linear action of $H$ on the vector space $T_oM$. When we let $H\subset \overline{H}$ act on the full tangent vector space $T_o\overline{M}$, the associated vector bundle is $T\overline{M} \vert_{M}$ and $D$ is the induced linear connection. In particular, we can give an expression of this connection as follows.

\begin{proposition}\label{prop:1}
    Let $M=G/H \subset \overline{M}$ be a reductive extrinsically homogeneous submanifold with a reductive decomposition $\g = \h + \m$. Then, the canonical $G$-connection $D$ at $o = [e]_H$ is given by the condition
    \begin{equation}
    \label{carcon2}
        (D _{X ^ *} B ^*) _o = - [X, B] ^* _o,
    \end{equation}
    for all $X \in \m$ and $B \in \overline{\g}$.
\end{proposition}

\begin{proof}
    Given $X\in \m$, the curve $\exp(tX)$ is horizontal in the bundle $G\to G/H$ with respect to the principal connection defined by $\m$. Therefore, the parallel transport along the curve \(c(t) = \exp(tX)\cdot o\) in $M$ is $(L_{\exp(tX)})_*$ for both the induced connection in the associated bundle $TM$ (this is a classical result, see~\cite[p.~192, Cor.~2.5]{KN1969}) and the associated bundle $T\overline{M}|_M$. Thus, we obtain:
\begin{align*}
  (D_{X^*} B^*)_o 
  &= \lim_{t \to 0} \frac{(L_{\exp(-tX)})_{*}\,B_{c(t)}^* \;-\; B_o^*}{t}\\
  &= \lim_{t \to 0} \frac{\bigl(\mathrm{Ad}_{\exp(-tX)}(B)\bigr)_o^* \;-\; B_o^*}{t}\\
  &= \left(\,\lim_{t \to 0} \frac{\mathrm{Ad}_{\exp(-tX)}(B) \;-\; B}{t}\right)_o^* = -\,[X,\,B]_o^*.
\end{align*}
Here, we used the fact that 
\((L_{\exp(-tX)})_{*}\,B_{\gamma(t)}^* 
= \bigl(\mathrm{Ad}_{\exp(-tX)}(B)\bigr)^*\),
since the differential of the left translation 
\(L_{\exp(-tX)}\) acts via the adjoint representation 
on the corresponding fundamental vector fields.
\end{proof}

\section{The Riemannian Case} \label{sec:Riemannian-Case}

Homogeneity in the case of Riemannian manifolds $(\overline{M},\overline{g})$ takes for granted that the group $\overline{G}$ acts transitively by isometries. From now on, we also assume that $\overline{G}$ acts efficiently, otherwise we consider $\overline{G}/\Gamma$ instead of $\overline{G}$, with $\Gamma=\{p\in \overline{G}:L_p=\mathrm{id}_{\overline{M}} \}$. We denote by $\overline{\nabla}^{\overline{g}}$ the Levi Civita connection of $(\overline{M},\overline{g})$.

As it is well known (for example, see~\cite[Prop.~1.4.8]{CC2019}), every homogeneous Riemannian manifold is reductive. For convenience in the following, we sketch the proof of this fact. Let $o=[e]_H\in\overline{M}$ (so that $\overline{h}$ is the isotropy subgroup of $o$). For every $X\in \overline{\mathfrak{g}}$, we consider the Kostant operator 
\begin{eqnarray*}
\overline{K}_X\colon T_o \overline{M} &\to &T_o \overline{M}\\
A_o&\mapsto &\overline\nabla^{\overline{g}} _{A_o}X^*=(\overline{\nabla}^{\overline{g}}_{X^*}A-\mathcal{L}_{X^*}A)_o,
\end{eqnarray*}
where $A\in\mathfrak{X}(\overline{M})$ is any extension of $A_o$. Since $\mathcal{L}_{X^*}\overline{g}=\overline{\nabla}^{\overline{g}} \overline{g}=0$, the Kostant operator is skew-symmetric, that is, $K_X\in\mathfrak{so}(T_o\overline{M},\overline{g}_o)\simeq \mathfrak{so}(n)$. We define
    \begin{equation}
        \label{inner}
        \overline{\phi} (X,Y) = - \overline{\mathfrak{B}} (\overline{K}_{X}, \overline{K}_{Y}), \quad X, Y \in \overline{\g},
    \end{equation}
where $\overline{\mathfrak{B}}$ is the Cartan-Killing metric in $\mathfrak{so}(n)$. We have that $\overline{\phi}$ is semidefinite positive. Furthermore, the restriction of $\overline{\phi}$ to $\overline{\h}$ is definite. Indeed, if $X\in \overline{\h}$ satisfies that $\overline{\phi} (X,X)=0$, then we have that $\overline{K}_{X}=(\overline{\nabla}^{\overline{g}} X^*)_o=0$, and since $X^*_o=0$, the Killing vector field vanishes, i.e., $X=0$. We now choose \[\overline{\m}=\overline{\h}^\perp=\{ X\in\overline\g :\: \overline{\phi} (X,Y)=0,\, \forall Y\in\overline{\h}\}.\]
We get a direct sum $\overline{\g}=\overline{\h}+\overline{\m}$ because $\overline{\m}\cap \overline{\h}=0$ (since $\overline{\phi}|_{\overline{\h}}$ is definite), and for each $X\in\overline{\g}$, we have that $X-\sum \overline{\phi}(X,U_i)U_i \in\overline{\m}$, where $\{U_i\}$ is an orthonormal basis of $\overline{\h}$. On the other hand, the adjoint invariance of $\mathfrak{B}$ easily yields $\mathrm{Ad}_{\overline{H}}\overline{\m}=\overline{\m}$, that is, we have a reductive decomposition. Obviously, not every reductive decomposition is defined with this procedure. Moreover, although this decomposition might seem canonical, it depends on the choice of the point $o\in \overline{M}$. If we pick other point $o'$, then the new isotropy and its Lie algebra $\overline{\h}'$ transforms under conjugation and the complement $\overline{\m}'$ is then different.

Let $G$ now be a Lie subgroup of $\overline{G}$ such that the orbit $M=G\cdot o$ is closed. The manifold $M$ is reductive extrinsically homogeneous as $G$ acts by isometries with respect to the metric $g=\overline{g}|_{TM}$. 
We obtain in the following a reductive decomposition of $\g$ induced in a natural way from any given initial reductive decomposition of $\overline{\g}=\overline{\h}+\overline{\m}$ of the ambient space.

\begin{theorem} \label{thm:main}
    Let $(\overline{M},\overline{g})$ be a homogeneous Riemannian manifold $\overline{M}=\overline{G}/\overline{H}$. Let $o=[e]_H\in\overline{M}$ and $M=G\cdot o$ a closed orbit defined by a Lie subgroup $G\subset \overline{G}$. We assume that $\overline{M}$ is equipped with a reductive decomposition $\overline{\g}=\overline{\h}+\overline{\m}$. Then 
    \[
    \g=\h + \m,\qquad \text{with}\qquad \m = \h ^\perp \cap \g,
    \]
is a reductive decomposition of $M=G/H$, $H=\overline{H}\cap G$, where $\h ^\perp =\{X\in \g:\: \overline{\phi}(X,Y)=0,\,\forall Y\in \h\}$.
\end{theorem}

\begin{proof}

If $X=X_{\overline{\m}} + X_{\overline{\h}}$ is the decomposition of an element $X \in \h^\perp$, we have that $X_{\overline{\h}}\in \h^\perp$, that is, $\h^{\perp} = \h ^{\perp _{\overline{\h}}} + \overline{\m}$, where $\h^{\perp _{\overline{\h}}}$ is the orthogonal to $\h$ in $\overline{\h}$ (recall that $\phi|_{\overline{\h}}$ is definite positive). Since $\h$ is $\mathrm{Ad}(H)$-invariant, also is $\h^{\perp _{\overline{\h}}}$, and as $\overline{\m}$ is invariant, we have the $\mathrm{Ad}(H)$-invariance of $\m$.

Obviously, $\h^{\perp} \cap \h = \{0\}$ since $\overline{\phi}|_{\overline{\h}}$ is definite positive. Finally, for any $X\in \g$, $X=X_{\overline{\m}} + X_{\overline{\h}}$, we decompose $X_{\overline{\h}}= X_{\h} + X_{\h^{\perp _{\overline{\h}}}}$ so that $X\in \h  + \m$ since $ X_{\h^{\perp _{\overline{\h}}}} + X_{\overline{\m}}\in (\h ^{\perp _{\overline{\h}}} + \overline{\m})\cap \g =  \h ^\perp \cap \g$. 
\end{proof}

\begin{remark}\label{rmk:3.2}
We can transfer the inner product from $(T_o \overline{M} , g_o)$ to $\overline{\m}$ through the bijection \eqref{eq:Identificación}. As the adjoint action of $\overline{H}$ on $\overline{\m}$ is transferred to the orthogonal action in $T_o \overline{M}$, the pull-back metric $\varphi_{\overline{\m}}^*g_o$ is $\mathrm{Ad}(\overline{H})$-invariant. We then define the bilinear form $\psi$ on $\overline{\g}$ as
\begin{equation}
    \label{Inn}
    \psi = \overline{\phi} |_{\overline{\h} \times \overline{\h}} + (\varphi_{\overline{\m}}^*g_o)| _{\overline{\m} \times \overline{\m}},
\end{equation}
which, by construction, is positive definite and $\mathrm{Ad}(\overline{H})$-invariant. Given the decompositions $\overline{\g}=\overline{\h}+\overline{\m}$ and $\g=\h+\m$ in the previous theorem, we can define the subspace $\n\subset \overline{\g}$ as the  orthogonal complement of $\overline{\h} + \m$ with respect to \eqref{Inn}. Thus, we have that $\overline{\g}=\overline{\h}+\m+\n$ and it is easy to verify that $\n$ is $\mathrm{Ad}(H)$-invariant. Geometrically, $(T_oM)^\perp=\{X^*_o:X\in\n\}$ and this invariance reflects the fact that $G$ preserves $TM^\perp$. However, in general, the space $(\m+\n)$ is not $\mathrm{Ad}(\overline{H})$-invariant and the decomposition $\overline{\g}=\overline{\h}+(\m+\n)$ cannot be regarded as a reductive decomposition.
\end{remark}

We now analyze the decomposition $\g = \h + \m$ of~\ref{thm:main} when the reductive decomposition $\overline{\g}=\overline{\h}+\overline{\m}$ is the one at the beginning of the section, that is, the decomposition (see~\cite[Prop.~1.4.8]{CC2019}) with $\overline{\m}=\overline{\h} ^{\perp}$ where the perpendicularity is defined by~\eqref{inner}. 
To avoid confusions, we write by $\overline{\m}=\overline{\h}^{\perp _{\overline{\phi}}}$. 
First, for $X\in \overline{\h}$ and $Y\in \overline{\g}$, we observe that the Kostant operator satisfies $\overline{K}_X(Y^*_o)=(\overline{\nabla}^{\overline{g}}_{X^*}Y^*+[Y,X]^*)_o=[Y,X]^*_o$, thus it coincides with the isotropy representation on $T_o\overline{M}$. On the other hand, for $X\in \g$ we have $\overline{K}_X|_{T_oM} = K_X + \mathrm{II}_o(X,\cdot)$, where $\mathrm{II}$ is the second fundamental form of $M$. With respect to an adapted orthonormal basis $\{u_1,\dots ,u_m;\ u_{m+1},\ldots ,u_n\}$ of $T_o\overline{M}$ such that $\mathrm{span}\{u_1,\ldots ,u_m\}=T_oM$, we have the following matrix expression 
\[
\overline{K}_X=\left(\begin{array}{cc} K_X & * \\ \mathrm{II}(X^*, \, ) & *     
\end{array} \right),\qquad X\in \g.
\]
Moreover, if $M=G\cdot o$ is a principal orbit of the action of $G$ on $\overline{M}$, then
\[
\overline{K}_X=\left(\begin{array}{cc} K_X & 0 \\ 0 & 0     
\end{array} \right),\qquad X\in \h=\g\cap \overline{\h},
\]
since the slice representation of $H$ (i.e., the isotropy representation on the normal space $(T_oM)^\perp=\mathrm{span}\{u_{m+1},\ldots,u_n\}$) is trivial, see for example, cf.~\cite[Rmk.~1.2.7]{A2004}. Now we define $\phi (X,Y)=-\mathfrak{B}(K_X,K_Y)$, $\mathfrak{B}$ being the Cartan-Killing metric in $\mathfrak{so}(T_oM)\simeq \mathfrak{so}(m)$. Using the block description of $\overline{K}_X$ obtained above, we have that 
\begin{equation}
    \label{prin}
\phi(X,Y)=\overline{\phi}(X,Y),\qquad \forall X\in \h, \quad \forall Y\in \g.
\end{equation}
Therefore,
\[ 
\h ^{\perp _{\phi}}= \{ Y\in \g :\: \phi(X,Y)=0,\forall X\in \h\},
\]
coincides with $\h ^{\perp _{\overline{\phi}}} \cap \g = \m$.  
In other words, if we start from the decomposition defined by $\overline{\phi}$, then the decomposition determined by $\phi$ is the one stated in~\ref{thm:main}. Note that if $M=G\cdot o$ is not a principal orbit, the isotropy representation on $T_oM^\perp$ is not trivial, we do not have \eqref{prin}, and $\h^{\perp _\phi}$ need not be $\m$.

\section{Homogeneous Structures of Submanifolds}\label{sec:hom-structures}

Let $(\overline{M}= \overline{G}/ \overline{H},\overline{g})$ be a homogeneous Riemannian manifold equipped with a reductive decomposition and let $\tilde{\nabla}$ be the associated canonical connection. We then have
\[
\tilde{\nabla} \overline{S} = 0, \quad \tilde{\nabla} \overline{R} = 0, \quad \tilde{\nabla} \overline{g} =0,
\]
where $\overline{S} = \overline{\nabla} ^g - \tilde{\nabla}$. Let $M \subset \overline{M}$ be a closed reductive extrinsically homogeneous submanifold. According to~\ref{thm:Esc}, there exists a $\overline{G}$-connection $D$ on $E := T\overline{M} \vert_{M}$ such that $TM \subset E$ is a $D$-parallel subbundle of $E$, and $\Gamma = \tilde{\nabla} - D$ is $D$-parallel.

\begin{lemma}\label{lem:lemurcillo}
    The following two conditions are equivalent:
    \begin{enumerate}
        \item $\Gamma = \tilde{\nabla} - D$ is $D$-parallel.
        \item $S = \overline{\nabla} ^g - D$ is $D$-parallel.
    \end{enumerate}
\end{lemma}

\begin{proof}
    Since $\tilde{\nabla}$ is the canonical connection, it follows that $\overline{G}$ coincides with the transvection group of $\tilde{\nabla}$, see~\cite[Thm.~I.25]{K1980}. Therefore, $\overline{S}$ is $\overline{G}$-invariant,  
    since $\tilde{\nabla} \overline{S} =0$. In particular, it is $G$-invariant. According to~\cite[Prop.~1.4.15]{CC2019}, we have $D \overline{S} = 0$. Finally, using this last fact, it follows that
    \[
    D \Gamma =D(\tilde{\nabla} - D) = D(\tilde{\nabla} - \overline{\nabla}^g + \overline{\nabla}^g - D) = D(\tilde{\nabla} - \overline{\nabla}^g) + D(\overline{\nabla}^g - D) = D S
    \]
    which proves the result.
\end{proof}

This result leads to the following definition which, in fact, generalizes the definition given in~\cite[Rmk.~6.1.5]{BCO2016} for homogeneous structures of submanifolds in spaces forms.

\begin{definition}
    Let $M \subset \overline{M}$ be a closed submanifold and let $S \in \Gamma(T^*M \otimes \mathrm{End}(T\overline{M}))$ be a tensor field. We say $S$ is a \emph{homogeneous structure} of $M$ if $D = \overline{\nabla} ^g -S$ is a metric connection and satisfies,
    \begin{itemize}
    \item[(a)] $TM \subset T\overline{M}|_{M}$ is a $D$-parallel subbundle of $T\overline{M}|_{M}$;
    \item[(b)] $S = \overline{\nabla} - D$ is $D$-parallel.
\end{itemize}
\end{definition}

In fact, if we combine~\ref{lem:lemurcillo} and~\ref{thm:Esc}, then we obtain the following corollary.

\begin{corollary}
    If $M \subset \overline{M}$ is an extrinsically homogeneous Riemannian submanifold then $M$ admits a homogeneous structure.
\end{corollary}

However, the crucial question here is the converse. To conclude, we propose the following problem:

\begin{problem}
What are the minimal conditions ensuring that a connected submanifold $M$ of  a homogeneous Riemannian manifold $\overline{M}$, endowed with a homogeneous structure, is an open subset of a homogeneous submanifold.
\end{problem}

Of course, if we ask $S$ to be $\overline{G}$-invariant, then, by~\cite[Thm.~1]{E1998}, the problem is resolved. This question aims to generalize the following theorem.

\begin{theorem}
\cite[Thm.~6.1.12]{BCO2016}
    A connected submanifold $M$ of a space form $\overline{M}$ is an open subset of a homogeneous submanifold of $\overline{M}$ if and only if $M$ admits a homogeneous structure.
\end{theorem}

Notice that, in this result, the authors do not assume the $\overline{G}$-invariance of $S$. Indeed, they show that, for submanifolds of space forms, no further conditions beyond the existence of a homogeneous structure are required.
\vspace{2mm}

We finally analyze the tensor $\Gamma$ when the reductive decomposition of the orbit $M=G\cdot o$ is the induced one by a reductive decomposition $\overline{\g}=\overline{\h}+\overline{\m}$ for $\overline{M}$ as in \ref{thm:main}. Indeed from \eqref{carcon} and \eqref{carcon2}, given $u\in T_oM$, we have that
\[
\Gamma _u (B^*)_o= (\tilde{\nabla}_uB^*-D_uB^*)_o=[X^u_\m - X^u_{\overline{\m}},B]^*_o,
\]
where $X^u_{\overline{\m}}\in \overline{\m}$, $X^u_\m \in \m$ satisfy $(X^u_{\overline{\m}})^*_o = (X^u_{\m})^*_o=u$. Then we have $X^u_\m - X^u_{\overline{\m}}\in \overline{\h}$, in fact
\[
X^u_\m - X^u_{\overline{\m}} \in \h ^{\perp _{\overline{\h}}}.
\]
In other words, the tensor $\Gamma = \tilde{\nabla}-D:TM\to \mathrm{End}(E)$ is defined by the adjoint representation of the normal subspace $\h ^{\perp _{\overline{\h}}}\subset \overline{\h}$. Obviously, the element $X^u_\m - X^u_{\overline{\m}} \in \h ^{\perp _{\overline{\h}}}$ depends on the choice of $\overline{\m}$ and $\g\subset\overline{\g}$, so does the way the action of $\h ^{\perp _{\overline{\h}}}$ is defined.

\section{Examples}\label{sec-examples}

\subsection{Horospheres in $\R \mathrm{H}(n)$} 

The real hyperbolic space $\RH(n)$ is a symmetric space with isometry group $\overline{G} = \mathrm{SO}(n,1)$. This action carries the Cartan decomposition (in the Lie algebra level)
\[
\overline{\g} = \overline{\h} + \overline{\m}, \qquad \overline{\g} = \mathfrak{so}(n,1), \quad \overline{\h} = \mathfrak{so}(n).
\]
The Cartan decomposition satisfies
\[
[\overline{\h}, \overline{\h}] \subset \overline{\h},\quad [\overline{\h}, \m] \subset \m, \quad [\m, \m] = \overline{\h}.
\]
An explicit description of the elements of this decomposition is,
\begin{equation*}
\overline{\g}=
\begin{Bmatrix}
\begin{pmatrix}
B & v_1 & v_2 \\
-v_2 ^t & a & 0 \\
-v_1 ^t& 0 & -a \\
\end{pmatrix} :\:
\begin{matrix}
B \in \mathfrak{so}(n-1);\\
v_1, v_2 \in \R ^{n-1}; \\ a \in \R
\end{matrix}
\end{Bmatrix},
\quad 
\overline{\h}=
\begin{Bmatrix}
\begin{pmatrix}
B & v & v \\
-v^t & 0 & 0 \\
-v^t & 0 & 0 \\
\end{pmatrix} :\:
\begin{matrix}
B \in \mathfrak{so}(n-1); \\
v \in \R^{n-1}
\end{matrix}
\end{Bmatrix}
\end{equation*}
and
\begin{equation*}
	\overline{\m}=
	\begin{Bmatrix}
		\begin{pmatrix}
			0 & v & -v \\
			v^t & a & 0 \\
			-v^t & 0 & -a \\
		\end{pmatrix} :\:
		\begin{matrix}
			v \in \R^{n-1};\; a\in \R
		\end{matrix}
	\end{Bmatrix}.
\end{equation*}
We make use of the description of $\R \mathrm{H}(n)$ as the warped product $(\R\times_f \R ^{n-1}, g = dt^2 + f(t) ^2 g_{\R^{n-1}})$, where $f(t) = e^{-t}$. 

 We now consider $G = \mathrm{SO}(n-1) \ltimes \R ^{n-1}$, which acts exclusively on the second factor of the warped product. The action of $G$ induces a foliation of $\R \mathrm{H}(n)$ which is known as \emph{Horosphere foliation}, the leaves of which are $L_t = \{ t \} \times \R ^{n-1}$, for all $t \in \R$.
If we consider de decomposition
\[
\g = \h + \m
\]
where
\begin{equation*}
\h= \mathfrak{so}(n-1)=
\begin{Bmatrix}
\begin{pmatrix}
B & 0 & 0 \\
0 & 0 & 0 \\
0 & 0 & 0 \\
\end{pmatrix} :\:
\begin{matrix}
B \in \mathfrak{so}(n-1)
\end{matrix}
\end{Bmatrix},
\end{equation*}
and
\begin{equation*}
	\m=
	\begin{Bmatrix}
		\begin{pmatrix}
			0 & 0 & -v \\
			v^t & 0 & 0 \\
			0 & 0 & 0 \\
		\end{pmatrix} :\:
		\begin{matrix}
			v \in \R^{n-1}
		\end{matrix}
	\end{Bmatrix} \subset \overline{\h},
\end{equation*}
one can check that $\m = \h ^{\perp _{\overline{\phi}}} \cap \g$ as in \ref{thm:main}. If we now follow~\ref{rmk:3.2}, then we obtain an $\mathrm{Ad}(H)$-decomposition,
\[
\overline{\g} = \overline{\h} + \m + \n, \quad \text{where }\n = \mathbb{R}\,\mathrm{diag} (0,1,-1).
\]

Furthermore, since the flow of any fundamental vector field $X^*_p = X$ with $X \in \R ^{n-1}$ preserves any orthonormal frame of $\R \mathrm{H}(n)$, from \ref{prop:1} it follows that the connection $D$ (given in~\ref{thm:Esc}) on $E = T\R \mathrm{H}(n)|_{\R ^{n-1}}$ coincides with $\nabla^\R + \nabla ^{\R^{n-1}}$ restricted to $E$ where $\nabla^\R$ and $\nabla ^{\R^{n-1}}$ are the Levi-Civita connections of $\R$ and $\R^{n-1}$ (with the Euclidean metric), respectively.

\subsection{Concentric spheres in $\mathbb{R}^{2m}$}
In this example, we broaden the realm of Riemannian actions to include conformal geometry. More precisely, we consider the punctured Eucliden space $(\R ^n-\{0\}, g_{\R^n})$, which is isometric to the warped product $(\R^+ \times_f \mathbb{S}^{n-1},dr^2 + f(r)^2 g_{\mathbb{S}^{n-1}})$, where $r$ is the radial coordinate (i.e., the distance from the origin), $f(r)=r$, and $g_{\mathbb{S}^{n-1}}$ is the round metric on the sphere. We denote this manifold by $(M,g)$. Although $(M,g)$ is not a homogeneous Riemannian manifold, it is a conformally homogeneous Riemannian manifold, meaning that there exists a Lie group $\overline{G}$ of conformal transformations that acts transitively on $(M, [g])$ ($[g]$ is the conformal class of $g$). Let $(\R^+, \, \cdot\,)$ be the Lie group of positive numbers with the action of the multiplication group. More specifically, the Lie group $\overline{G}$ is
\[
\mathrm{Conf}(M,[g]) = (\R ^+, \,\cdot \,) \times \mathrm{SO}(n),
\]
the action of $\overline{G}$ on $M$ is given by
\begin{align*}
    \left(\R^+ \times \mathrm{SO}(n)\right) \times  \left(\R^+ \times_f \mathbb{S}^{n-1}\right) &\to  \R^+ \times_f \mathbb{S}^{n-1} \\
    ((r, A),(s,p)) &\longmapsto (r \cdot s, A\cdot p)
\end{align*}
and the isotropy group $\overline{H}$ is $\mathrm{SO}(n-1)$. Note that $r \frac{\partial}{\partial r}$ is not a Killing vector field, but a conformal Killing vector field, that is, $\mathcal{L}_{r\frac{\partial}{\partial r}} g = 2g $. The Lie algebra $\overline{\g}$ of $\overline{G}$ admits a reductive decomposition 
\[
\overline{\g} = \overline{\h} + \overline{\m}
\]
where $\overline{\h} = \so (n-1)$ is the Lie algebra of $\overline{H}$ and $\m$ is an $\mathrm{Ad}(\overline{H})$-invariant subspace. An explicit description of these objects is
\begin{equation*}
\overline{\g}=
\begin{Bmatrix}
\begin{pmatrix}
0 & v & 0 \\
-v^t& B & 0 \\
 0 &0 &a
\end{pmatrix} :\:
\begin{matrix}
B \in \mathfrak{so}(n-1);\\
v \in \R ^{n-1};\; a \in \R
\end{matrix}
\end{Bmatrix}, 
\quad 
\overline{\h}=
\begin{Bmatrix}
\begin{pmatrix}
0 & 0 & 0 \\
0 &  B &0 \\
0 & 0  & 0 \\
\end{pmatrix} :\:
\begin{matrix}
B \in \mathfrak{so}(n-1)
\end{matrix}
\end{Bmatrix}
\end{equation*}
and
\begin{equation*}
	\overline{\m}=
	\begin{Bmatrix}
		\begin{pmatrix}
			0 & v & 0 \\
			-v^t & 0 & 0 \\
			0 & 0 & a \\
		\end{pmatrix} :\:
		\begin{matrix}
			v \in \R^{n-1};\; a\in \R
		\end{matrix}
	\end{Bmatrix}.
\end{equation*}

Let $G$ be any Lie group of isometries acting transitively on $\mathbb{S}^{n-1}$ and let $H$ be its isotropy group; see~\cite{AHL2023} for the explicit expressions of these Lie groups and its Lie algebras. Since $G \subset \overline{G}$ is closed, we can apply~\ref{thm:Esc} provided that $M = G/H$ admits a reductive decomposition. 
To verify this, we adapt here~\ref{thm:main}. Note that although $(M,g)$ is not a homogeneous Riemannian manifold, \ref{thm:main} uses only the fact that $\overline{\phi}$ is positive definite on $\overline{\h}$ and this property still holds in this example. We set $\m = \h ^{\perp _{\overline{\phi}}} \cap \g$ where $\g$ and $\h$ are the Lie algebras of $G$ and $H$, respectively. Consequently, analogously to~\ref{thm:main}, we obtain
\begin{equation}\label{eq:red}
    \g = \h + \m
\end{equation}
which is a reductive decomposition of $M=G/H$, $H=\overline{H}\cap G$, where $\h ^\perp =\{X\in \g:\: \overline{\phi}(X,Y)=0,\,\forall Y\in \h\}$. Furthermore, by~\ref{rmk:3.2}, we have an $\mathrm{Ad}(H)$-invariant decomposition,
\[
 \overline{\g} = \overline{\h} + \m + \n, \quad \text{where }\n = \R \, \mathrm{diag}(0,1).
\]
Let $\tilde{\nabla}$ be the canonical connection corresponding to the reductive decomposition~\eqref{eq:red}. Since $\frac{\partial}{\partial r}$ is invariant under $G$, it follows that the connection $D$ described in~\ref{prop:1} is $\nabla ^{\R ^+} + \tilde{\nabla}$ restricted to $E =  T \R ^{2m} |_{\mathbb{S}^{2m-1}}$, where $\nabla ^{\R ^+}$ is the Levi-Civita connection of $\R ^+$ with the Euclidean metric. 
As final remark, this construction of $\m$ depends on the explicit expression of $\g$ as a Lie subalgebra of $\Bar{\g}$. To obtain it, we again refer the reader to~\cite{AHL2023}.

\printbibliography

\end{document}